\documentclass[11pt]{article}

\usepackage[utf8]{inputenc} % For proper input encoding
\usepackage[T1]{fontenc}      % For proper font encoding
\usepackage{amsmath}        % For advanced math environments and commands
\usepackage{amssymb}        % For additional math symbols
\usepackage{amsthm}         % For theorem-like environments
\usepackage{graphicx}       % For including graphics
\usepackage{hyperref}       % For hyperlinks in the PDF
\usepackage{geometry}       % For customizing page layout

\title{An infinite family of overpartition congruences mod powers of 2}
\author{Zhumagali Shomanov \\
        Department of Mathematics \\
        University of Florida \\
        \texttt{zshomanov@ufl.edu}
        \and
        Frank Garvan \\
        Department of Mathematics \\
        University of Florida \\
        \texttt{fgarvan@ufl.edu}}
\date{\today}

\theoremstyle{plain}
\newtheorem{theorem}{Theorem}[section] 
\newtheorem{lemma}[theorem]{Lemma}

\newtheorem{corollary}[theorem]{Corollary}

\theoremstyle{definition}

\theoremstyle{remark}

\begin{document}

\maketitle

%----- ABSTRACT -----%
\begin{abstract}
We prove an infinite family of Hecke-like congruences for the overpartition function modulo powers of 2. Starting from a recent identity of Garvan and Morrow and iterating Atkin’s $U_2$ operator, we determine lower bounds on the 2-adic valuations of the coefficients that arise at each step. Our approach yields new modular equations relating the Hauptmoduln $G_2$ on $\Gamma_0(2)$ and $G_8$ on $\Gamma_0(8)$, together with explicit $U_2$–action formulas.
\end{abstract}

%----- TABLE OF CONTENTS (OPTIONAL) -----%
\tableofcontents

%----- INTRODUCTION -----%
\section{Introduction}
\label{sec:introduction}

An overpartition of a non-negative integer $n$ is a partition of $n$ in which the first occurrence of each distinct part may be overlined. For example, let's find the overpartitions of $n=3$. The ordinary partitions of 3 are $3$, $2+1$, and $1+1+1$. We apply the overlining rule (first occurrence of a part may be overlined):
\[
3,\overline{3},2+1,\overline{2}+1, 2+\overline{1}, \overline{2}+\overline{1},1+1+1,\overline{1}+1+1.
\]
Thus, there are 8 overpartitions of 3. So, $\overline{p}(3) = 8$.
The generating function for $\bar{p}(n)$, noted by Corteel and Lovejoy \cite{CorteelLovejoy}, is given by
$$ \sum_{n=0}^{\infty} \overline{p}(n) q^n = \prod_{k=1}^{\infty} \frac{1+q^k}{1-q^k}=\dfrac{\eta(2\tau)}{\eta(\tau)^2}=\varphi(\tau), $$
where $q=e^{2\pi i \tau}$ and $\tau\in \mathbb{H} = \{ \tau \in \mathbb{C} \mid \operatorname{Im} \tau > 0 \}$.

Overpartitions satisfy numerous congruences, particularly modulo powers of 2. Here are a few examples \cite{FJM}, \cite{HS}
\begin{align*}
\overline{p}(4n+3) &\equiv 0 \pmod 8, \\
\overline{p}(8n+7) &\equiv 0 \pmod{64},\\
\overline{p}(8191^4n+8191^3)&\equiv0\pmod{4096}.
\end{align*}
The last result follows from
\begin{equation}\label{OVP}
\sum_{n=0}^\infty \left(\ell^3 \overline{p}(\ell^2 n)+\ell\left(\dfrac{-n}{\ell}\right)\overline{p}(n)+\overline{p}\left(\dfrac{n}{\ell^2}\right)\right)q^n\equiv(\ell^3+1)\sum_{n=0}^\infty \overline{p}(n)q^n\pmod{2^{12}},
\end{equation}
which, in turn, follows from a more general result of Garvan and Morrow~\cite{GarvanMorrow}. Here $\left(\frac{a}{b}\right)$ denotes the Jacobi symbol and $\ell$ is an odd prime.

The main result of this paper is
\begin{equation}\label{OVPinfIntro}
\ell^3 \overline{p}(2^\alpha\ell^2 n)+\ell\left(\dfrac{-2^\alpha n}{\ell}\right)\overline{p}(2^\alpha n)+\overline{p}\left(\dfrac{2^\alpha n}{\ell^2}\right)\equiv (\ell^3+1)\overline{p}(2^\alpha n)\pmod{2^{\alpha+12}}
\end{equation}
for all $n\ge0$, where $\ell$ is an odd prime and $\alpha\ge0$.

If we replace $n$ by $\ell n$ in (\ref{OVPinfIntro}), require $\ell\nmid n$, and $\ell\equiv-1\pmod{2^{\alpha+12}}$, we will obtain
\begin{equation*}
\overline{p}(2^\alpha\ell^3n)\equiv0\pmod{2^{\alpha+12}}.
\end{equation*}

\section{Modular Forms}

The full modular group is defined as
\[
\mathrm{SL}_2(\mathbb{Z}) = \left\{ \begin{pmatrix} a & b \\ c & d \end{pmatrix} : a,b,c,d \in \mathbb{Z},\; ad-bc = 1 \right\}.
\]
This group acts on the upper half-plane
\[
\mathbb{H} = \{ \tau \in \mathbb{C} \mid \operatorname{Im}\tau > 0 \}
\]
by linear fractional (Möbius) transformations
\[
\gamma \tau = \frac{a\tau+b}{c\tau+d}, \quad \text{for } \gamma = \begin{pmatrix} a & b \\ c & d \end{pmatrix} \in \mathrm{SL}_2(\mathbb{Z}).
\]
We will consider the subgroup $\Gamma_0(N)$ which is defined as 
\[
\Gamma_0(N) = \left\{ \begin{pmatrix} a & b \\ c & d \end{pmatrix} \in \mathrm{SL}_2(\mathbb{Z}) : c \equiv 0 \pmod{N} \right\}
\]
A \emph{cusp} of $\Gamma_0(N)$ is an equivalence class of points in $\mathbb{Q} \cup \{\infty\}$ under the action of $\Gamma_0(N)$. For each cusp $r$ of $\Gamma_0(N)$, choose $\sigma_r \in \mathrm{SL}_2(\mathbb{Z})$ with $\sigma_r \infty = r$. Let the cusp width be
\[
w_r = \min\left\{\, w>0:
\sigma_r^{-1}\begin{pmatrix}1&1\\[2pt]0&1\end{pmatrix}^w\sigma_r\in \Gamma_0(N) \right\}.
\]

A \emph{modular function} on $\Gamma_0(N)$ is a meromorphic function $f : \mathbb{H} \to \mathbb{C}$ satisfying:
\begin{enumerate}
    \item For all $\gamma = \begin{pmatrix} a & b \\ c & d \end{pmatrix}\in \Gamma_0(N)$,
    \[
    f(\gamma \tau) = f(\tau).
    \]   
    \item The function $f$ is meromorphic on the upper half-plane.  
    \item Let $r$ be a cusp of $\Gamma_0(N)$, and $\sigma_r$ be such that $\sigma_r\infty=r$. Let $q_r = e^{2\pi i \tau / w_r}$. Then
    \[
    f(\sigma_r\tau)=\sum_{n=n_0}^{\infty} a_n q_r^{n}
    \]
    with only finitely many negative indices $n$. The order at $r$ is
    \[
    \operatorname{Ord}(f,r,\Gamma_0(N)) \;=\; \min\{\, n \in \mathbb{Z} : a_n \neq 0 \,\}.
    \]
\end{enumerate}

A central role in our study is played by the Dedekind eta function, defined for $\tau\in\mathbb{H}$ by
\[
\eta(\tau) = e^{\pi i\tau/12} \prod_{n=1}^{\infty} (1-e^{2\pi in\tau}).
\]

It is well known that $\eta(\tau)$ is holomorphic and non-vanishing on $\mathbb{H}$. Moreover, $\eta(\tau)$ satisfies the transformation law
\[
\eta(\gamma \tau) = \nu_\eta(\gamma) \, (c\tau+d)^{1/2} \, \eta(\tau),
\]
for all $\gamma = \begin{pmatrix} a & b \\ c & d \end{pmatrix}\in \mathrm{SL}_2(\mathbb{Z})$, where, for $c>0$~\cite{Knopp2008}
\[
\nu_\eta(\gamma) = \begin{cases}
\left(\frac{d}{c}\right) e^{2\pi i((a+d)c - bd(c^2-1) - 3c)/24}, & \text{if } c \text{ is odd}, \\
\left(\frac{c}{d}\right) e^{2\pi i((a+d)c - bd(c^2-1) + 3d - 3 - 3cd)/24}, & \text{if } c \text{ is even}.
\end{cases}
\]
A function $f(\tau)$ of the form
\[
f(\tau)
=
\prod_{\delta\mid N}
\eta\bigl(\delta\tau\bigr)^{m_\delta},
\]
where \(N\) is a positive integer and each exponent \(m_\delta\) is an integer is called an eta-quotient.

\begin{theorem}[Newman~\cite{NewmanEtaQuotient}]\label{thm:NewmanEtaQuotient}
Let 
\begin{equation}\label{eq:eta-quotient}
f(\tau)=\prod_{\delta\mid N}\eta(\delta \tau)^{m_\delta},
\end{equation}
where \(m_\delta\in\mathbb{Z}\).  Then \(f(\tau)\) is a modular function on \(\Gamma_0(N)\) if
\begin{enumerate}
  \item \(\displaystyle \sum_{\delta\mid N}\delta\,m_\delta\equiv0\pmod{24},\)
  \item \(\displaystyle \sum_{\delta\mid N}\frac{N}{\delta}\,m_\delta\equiv0\pmod{24},\)
  \item \(\displaystyle \prod_{\delta\mid N}\delta^{\,m_\delta}\) is a rational square,
  \item \(\displaystyle \sum_{\delta\mid N}m_\delta=0.\)
\end{enumerate}
\end{theorem}

\begin{theorem}[Ligozat~\cite{Ligozat}]\label{thm:Ligozat}
Let \(N\) be a positive integer, and consider the eta-quotient
\[
f(\tau)=\prod_{\delta\mid N}\eta(\delta \tau)^{m_\delta},
\qquad m_\delta\in\mathbb{Z}.
\]
Fix a cusp represented by the fraction \(r=\frac{c}{d}\) with \(\gcd(c,d)=1\) and \(d\mid N\).  Then the order of \(f\) at the cusp \(r\) is
\[
\operatorname{Ord}(f,r,\Gamma_0(N))=\frac{N}{24}
\sum_{\delta\mid N}
\frac{\gcd(d,\delta)^{2}m_{\delta}}{\gcd(d,\,N/d)d\delta}.
\]
\end{theorem}

\begin{theorem}[Gordon-Hughes~\cite{GordonHughes}]\label{thm:GordonHughes}
Let \(\ell\) be a prime dividing \(N\), and let $f$ be a modular function for $\Gamma_0(\ell N)$.
Denote by \(\pi(n)\) the \(p\)-adic valuation of \(n\).  Let $r=\frac{c}{d}$ be a cusp of \(\Gamma_0(N)\), where \(d|N\) and \(\gcd(c,d)=1\).  Then $f|U_\ell$ is a modular function on $\Gamma_0(N)$, and
\[
\operatorname{Ord}\bigl(f|U_\ell,r,\Gamma_0(N)\bigr)\ge
\begin{cases}
\frac{1}{\ell}\,\operatorname{Ord}(f,r/\ell,\Gamma_0(\ell N)),
&\operatorname{if} \pi(d)\ge \frac12\,\pi(N),\\
\operatorname{Ord}(f,r/\ell,\Gamma_0(\ell N)),
&\operatorname{if } 0<\pi(d)<\frac12\,\pi(N),\\
\displaystyle\min_{0\le\lambda\le \ell-1}\operatorname{Ord}(f,(r+\lambda)/\ell,\Gamma_0(\ell N)),
&\operatorname{if } \pi(d)=0.
\end{cases}
\]
\end{theorem}

For $f(\tau)=\sum_{n=0}^\infty a(n)q^n$, the Atkin $U_p$ operator is defined as
\[
f | U_p=\sum_{n=0}^\infty a(p n)q^{n}.
\]

%----- MAIN RESULTS -----%
\section{Main Results}
\label{sec:main_results}

The main result we want to prove is
\begin{theorem}\label{thm:OVPinf}
Let $\alpha\ge1$, and $\ell$ be an odd prime. Then the following is true for all $n\ge0$
\begin{equation}\label{OVPinf}
\ell^3 \overline{p}(2^\alpha\ell^2 n)+\ell\left(\dfrac{-2^\alpha n}{\ell}\right)\overline{p}(2^\alpha n)+\overline{p}\left(\dfrac{2^\alpha n}{\ell^2}\right)\equiv (\ell^3+1)\overline{p}(2^\alpha n)\pmod{2^{\alpha+12}}.
\end{equation}
\end{theorem}
Define
\[
\varphi_2(\tau)=\dfrac{\eta(2\tau)}{\eta(\tau)^2}=\Phi_2(q),
\]
and
\begin{equation}\label{HauptmodulGamma_0_2}
g_2(\tau)=\left(\dfrac{\eta(2\tau)}{\eta(\tau)}\right)^{24}=G_2(q).
\end{equation}
By Theorem~\ref{thm:NewmanEtaQuotient}, (\ref{HauptmodulGamma_0_2}) is a modular function on $\Gamma_0(2)$. Its only pole is a simple pole at the cusp zero for $\Gamma_0(2)$. We need the following
\begin{theorem}
Let $\ell$ be an odd prime. Then
\begin{multline}\label{OVPNewman}
\sum_{n=0}^\infty \left(\ell^3 \overline{p}(\ell^2 n)+\ell\left(\dfrac{-n}{\ell}\right)\overline{p}(n)+\overline{p}\left(\dfrac{n}{\ell^2}\right)\right)q^n=(\ell^3+1)\sum_{n=0}^\infty\overline{p}(n)q^n\\
+\sum_{s=1}^{(\ell^2-1)/8}a(j,\ell)2^{12s}\Phi_2(q)G_2^s(q),
\end{multline}
where $a(j,\ell)$ are integers.
\end{theorem}
\begin{proof}
The result follows from the proof of Theorem 1.5 given in Section 3 of~\cite{GarvanMorrow}.
\end{proof}

We prove Theorem~\ref{thm:OVPinf} by repeatedly applying the Atkin $U_2$ operator to (\ref{OVPNewman}). However, obtaining an explicit formula for $\Phi_2(q)G_2^s(q)|U_2^\alpha$ is difficult. Instead, observe that
\[
\Phi_2(q)G_2^s(q)|U_2=\Phi_2(-q)G_2^s(-q)|U_2.
\]
It turns out that finding a formula for $\Phi_2(-q)G_2^s(-q)|U_2$ is much easier. Hence, in what follows we develop an explicit formula for $\Phi_2(-q)G_2^s(-q)|U_2^\alpha$.
\begin{lemma}
\begin{equation}\label{etanegative}
\eta(2\tau+1)=e^{\pi i/12}\eta(2\tau),\qquad\eta\left(\tau+\dfrac12\right)=e^{\pi i/24}\dfrac{\eta^3(2\tau)}{\eta(\tau)\eta(4\tau)}.
\end{equation}
\end{lemma}
\begin{proof}
The first formula follows directly from the definition of the eta function. For the second formula,
\begin{align}
\eta\left(\tau+\dfrac12\right)&=e^{\pi i/24}e^{\pi i \tau/12}\prod_{n=1}^\infty\left(1-e^{\pi in}e^{2\pi i n\tau}\right)=e^{\pi i/24}q^{1/24}\prod_{n=1}^\infty\left(1-(-1)^nq^n\right)\nonumber \\
&=e^{\pi i/24}q^{1/24}\prod_{n=1}^\infty\left(1+q^{2n-1}\right)\left(1-q^{2n}\right)=e^{\pi i/24}q^{1/24}\prod_{n=1}^\infty\dfrac{\left(1+q^{n}\right)\left(1-q^{2n}\right)}{1+q^{2n}}.\label{eq:tau12}
\end{align}
Next, note that
\[
\prod_{n=1}^\infty(1+q^n)=\prod_{n=1}^\infty\dfrac{1-q^{2n}}{1-q^n},\qquad\prod_{n=1}^\infty(1+q^{2n})=\prod_{n=1}^\infty\dfrac{1-q^{4n}}{1-q^{2n}}.
\]
Thus, we can rewrite (\ref{eq:tau12}) as
\[
\eta\left(\tau+\dfrac12\right)=e^{\pi i/24}q^{1/24}\prod_{n=1}^\infty\dfrac{(1-q^{2n})^3}{(1-q^n)(1-q^{4n})}=e^{\pi i/24}\dfrac{\eta^3(2\tau)}{\eta(\tau)\eta(4\tau)}.
\]
\end{proof}
\begin{lemma}\label{PhiGnegative}
\begin{equation}
\Phi_2(-q)=\dfrac{\eta^2(\tau)\eta^2(4\tau)}{\eta^5(2\tau)},\qquad G_2(-q)=-\dfrac{\eta^{24}(\tau)\eta^{24}(4\tau)}{\eta^{48}(2\tau)}
\end{equation}
\end{lemma}
\begin{proof}
Since $q=e^{2\pi i \tau}$, we have $-q=e^{2\pi i(\tau+1/2))}$. Thus,
\[
\Phi_2(-q)=\dfrac{\eta(2\tau+1)}{\eta^2(\tau+1/2)}=\dfrac{\eta^2(\tau)\eta^2(4\tau)}{\eta^5(2\tau)},
\]
where the last equality follows from~(\ref{etanegative}). Similarly,
\[
G_2(-q)=\dfrac{\eta^{24}(2\tau+1)}{\eta^{24}(\tau+1/2)}=-\dfrac{\eta^{24}(\tau)\eta^{24}(4\tau)}{\eta^{48}(2\tau)}.
\]
\end{proof}
Let
\[
g_8(\tau)=\dfrac{\eta^4(8\tau)\eta^2(2\tau)}{\eta^2(4\tau)\eta^4(\tau)}=G_8(q).
\]
By Theorem~\ref{thm:NewmanEtaQuotient} $G_8(q)$ is a modular function on $\Gamma_0(8)$, and by Theorem~\ref{thm:Ligozat}, it has the only pole at the cusp zero for $\Gamma_0(8)$ of order one, and therefore is a Hauptmodul on $\Gamma_0(8)$.

We will need the following modular equations
\begin{lemma}
\begin{equation}
G_2^2(q)=G_2(q^2)+G_2(q)(48G_2(q^2)+4096G_2^2(q^2))\label{modeq1}
\end{equation}
\end{lemma}
\begin{proof}
Note that $G_2(-q)$ and $G_2(q)$ are roots of~(\ref{modeq1}). Consider
\begin{equation}\label{eq:quad1}
u^2+bu+c=(u-G_2(-q))(u-G_2(q))=u^2-(G_2(-q)+G_2(q))u+G_2(-q)G_2(q).
\end{equation}
We want to show that 
\begin{align}
b&=-(48G_2(q^2)+4096G_2^2(q^2))\label{coeff:b},\\
c&=-G_2(q^2)\label{coeff:c}.
\end{align}
The formula~(\ref{coeff:c}) follows easily from the definition of $G_2(q)$ and Lemma~\ref{PhiGnegative}.

By applying Theorem~\ref{thm:GordonHughes}, we see that $G_2(q)|U_2$ is a modular function on $\Gamma_0(2)$, and its only  pole is at the cusp zero for $\Gamma_0(2)$, of order two. This means
\[
G_2(q)|U_2-a_1G_2(q)-a_2G_2^2(q)
\]
has no poles and must be a constant. By comparing the coefficients we find that the constant is zero, $a_1=24$, and $a_2=2048$. Therefore,
\begin{equation}\label{eq:GU1}
G_2(q)|U_2=24G_2(q)+2048G_2^2(q).
\end{equation}
On the other hand, an elementary calculation shows that
\begin{equation}\label{eq:GU2}
G_2(q)|U_2=\dfrac12(G_2(-\sqrt{q})+G_2(\sqrt{q})).
\end{equation}
By comparing~(\ref{eq:GU1}) and (\ref{eq:GU2}), and replacing $q$ by $q^2$, we obtain (\ref{coeff:b}). Thus, (\ref{modeq1}) holds.
\end{proof}
\begin{lemma}\label{lemma:modeq2}
\begin{equation}
G_8^2(q)=(8G_8(q^2)+32G_8^2(q^2))G_8(q)+(G_8(q^2)+4G_8^2(q^2))\label{modeq2}
\end{equation}
\end{lemma}
\begin{proof}
The proof is similar to the proof of the previous lemma. Note that $G_8(-q)$ and $G_8(q)$ are roots of~(\ref{modeq2}). Consider
\begin{equation}\label{eq:quad2}
u^2+bu+c=(u-G_8(-q))(u-G_8(q))=u^2-(G_8(-q)+G_8(q))u+G_8(-q)G_8(q).
\end{equation}
We want to show that 
\begin{align}
b&=-(8G_8(q^2)+32G_8^2(q^2))\label{coeff:b2},\\
c&=-(G_8(q^2)+4G_8^2(q^2))\label{coeff:c2}.
\end{align}

By applying Theorem~\ref{thm:GordonHughes}, we see that  $G_8(q)|U_2$ is a modular function on $\Gamma_0(8)$, and its only  pole is at the cusp zero for $\Gamma_0(8)$, of order two. This means
\[
G_8(q)|U_2-a_1G_8(q)-a_2G_8^2(q)
\]
has no poles and must be a constant. By comparing the coefficients we find that the constant is zero, $a_1=4$, and $a_2=16$. Therefore,
\begin{equation}\label{eq:GU81}
G_8(q)|U_2=4G_8(q)+16G_8^2(q).
\end{equation}
On the other hand,
\begin{equation}\label{eq:GU82}
G_8(q)|U_2=\dfrac12(G_8(-\sqrt{q})+G_8(\sqrt{q})).
\end{equation}
By comparing~(\ref{eq:GU81}) and (\ref{eq:GU82}), and replacing $q$ by $q^2$, we obtain (\ref{coeff:b2}).

To prove~(\ref{coeff:c2}), we need to calculate $G_8(-q)G_8(q)$. Using the definition of $G_8(q)$ and Lemma~\ref{etanegative}, we obtain
\[
G_8(-q)=-\dfrac{\eta^4(8\tau)\eta^2(4\tau)\eta^4(\tau)}{\eta^{10}(2\tau)},
\]
and therefore,
\[
G_8(-q)G_8(q)=-\dfrac{\eta^8(8\tau)}{\eta^8(2\tau)}.
\]
By Theorem~\ref{thm:NewmanEtaQuotient}, the function
\[
\left(\dfrac{\eta(4\tau)}{\eta(\tau)}\right)^8
\]
is a modular function on $\Gamma_0(8)$, and its only pole is at the cusp 0 for $\Gamma_0(8)$, of order 2. Thus,
\[
\left(\dfrac{\eta(4\tau)}{\eta(\tau)}\right)^8-a_1G_8(q)-a_2G_8^2(q)
\]
must be a constant. By comparing the coefficients, we obtain
\begin{equation}\label{eq:G4G8}
\left(\dfrac{\eta(4\tau)}{\eta(\tau)}\right)^8=G_8(q)+4G_8^2(q).
\end{equation}
After replacing $q$ by $q^2$ we obtain
\[
G_8(-q)G_8(q)=-(G_8(q^2)+4G_8^2(q^2)).
\]
This completes the proof of (\ref{coeff:c2}) and the lemma.
\end{proof}
\begin{lemma}
\begin{equation}\label{modeq3}
G_2(q)=G_8(q)+20G_8^2(q)+128G_8^3(q)+256G_8^4(q)
\end{equation}
\end{lemma}
\begin{proof}
By Theorem~\ref{thm:Ligozat}, the only pole of $G_2(q)$, considered as a modular function on $\Gamma_0(8)$, is at the cusp zero for $\Gamma_0(8)$, of order four. This means
\[
G_2(q)-a_1G_8(q)-a_2G_8^2(q)-a_3G_8^3(q)-a_4G_8^4(q)
\]
must be a constant. By comparing the coefficients, we get~(\ref{modeq3}).
\end{proof}
\begin{lemma}
\begin{align}
G_8(q)|U_2 &= 4G_8(q) + 16G_8^2(q) \label{eq:G8U} \\
\Phi_2(-q)|U_2 &= (1 + 4G_8(q))\Phi_2(-q^2) \label{eq:Phi2U} \\
(\Phi_2(-q)G_8(q))|U_2 &= (2G_8(q) + 8G_8^2(q))\Phi_2(-q^2) \label{eq:Phi2G8U} \\
(\Phi_2(-q)G_2(-q))|U_2 &= P(G_8(q))\Phi_2(-q^2) \label{eq:Phi2G2U}
\end{align}
where $P(x)$ is the polynomial
\begin{align*}
P(x) = 2\cdot13x &+ 2^3\cdot403x^2 + 2^6\cdot2015x^3 + 2^{10}\cdot2431x^4 \\
&+ 2^{15}\cdot819x^5 + 2^{19}\cdot325x^6 + 2^{22}\cdot151x^7 + 2^{26}\cdot19x^8 + 2^{30}x^9.
\end{align*}
\end{lemma}
\begin{proof}
We already proved (\ref{eq:G8U}) since it is (\ref{eq:GU81}) from the proof of Lemma~\ref{lemma:modeq2}.

Let $F(q) = \dfrac{\Phi_2(-q)}{\Phi_2(-q^4)}$. Using Lemma~\ref{PhiGnegative},
\[
F(q)=\dfrac{\eta^2(\tau)\eta^5(8\tau)}{\eta^5(2\tau)\eta^2(16\tau)}.
\]
By Theorem~\ref{thm:NewmanEtaQuotient}, $F(q)$ is a modular function on $\Gamma_0(16)$, and by Theorem~\ref{thm:GordonHughes}, $F(q)|U_2$ is a modular function on $\Gamma_0(8)$, and has a simple pole at the cusp 0 for $\Gamma_0(8)$. Thus,
\[
F(q)|U_2-a_1G_8
\]
must be a constant. By comparing coefficients we see that 
\[
F(q)|U_2=1+4G_8.
\]
On the other hand, 
\[
F(q)|U_2=\cfrac{\Phi_2(-q)}{\Phi_2(-q^4)}|U_2=\cfrac{\Phi_2(-q)|U_2}{\Phi_2(-q^2)}
\]
by using the properties of the $U_2$ operator. Thus,
\[
\Phi_2(-q)|U_2=(1+4G_8(q))\Phi_2(-q^2).
\]

Proofs of (\ref{eq:Phi2G8U}) and (\ref{eq:Phi2G2U}) are similar.
\end{proof}
\begin{theorem}\label{FG2Upowers}
For $s\ge0$ there exist integers $c_{1r}^s$ such that
\begin{equation}\label{FG2U2}
\Phi_2(-q)G_2^s(-q)|U_2=\Phi_2(-q^2)\sum_{r=\lceil s/2\rceil}^{8s+1} c_{1r}^sG_8^r
\end{equation}
and
\begin{equation}\label{FG2U2coeffval}
\pi_2(c_{1r}^s)\ge2r-s
\end{equation}
\end{theorem}
\begin{proof}
We proceed by induction on $s$. Formulas~(\ref{eq:Phi2U}), (\ref{eq:Phi2G2U}) show that the result is true for $s=0$ and $s=1$. Assume $k\ge1$ is fixed and the result is true for $s\le k$. We prove that it is true for $s=k+1$.

Replace $q$ by $-q$ in~(\ref{modeq1}), and then every occurrence of $G_2(q^2)$ in terms of $G_8(q^2)$ using~(\ref{modeq3}), to obtain
\begin{multline*}
G_2^2(-q)=G_{8}(q^2) + 2^{2} \cdot 5 G_{8}^{2}(q^2) + 2^{7} G_{8}^{3}(q^2) + 2^{8} G_{8}^{4}(q^2)\\
+ G_2(-q)\Big(48\left(G_{8}(q^2) + 2^{2} \cdot 5 G_{8}^{2}(q^2) + 2^{7} G_{8}^{3}(q^2) + 2^{8} G_{8}^{4}(q^2)\right)\\
+4096\left(G_{8}(q^2) + 2^{2} \cdot 5 G_{8}^{2}(q^2) + 2^{7} G_{8}^{3}(q^2) + 2^{8} G_{8}^{4}(q^2)\right)^2\Big)
\end{multline*}
After multiplying by $\Phi_2(-q)G_2^{k-1}(-q)$, we get
\begin{multline*}
\Phi_2(-q)G_2^{k+1}(-q)\\
=\Phi_2(-q)G_2^{k-1}(-q)\Big(G_{8}(q^2) + 2^{2} \cdot 5 G_{8}^{2}(q^2) + 2^{7} G_{8}^{3}(q^2) + 2^{8} G_{8}^{4}(q^2)\Big)\\
\qquad + \Phi_2(-q)G_2^k(-q)\Big(48\left(G_{8}(q^2) + 2^{2} \cdot 5 G_{8}^{2}(q^2) + 2^{7} G_{8}^{3}(q^2) + 2^{8} G_{8}^{4}(q^2)\right)\\
+4096\left(G_{8}(q^2) + 2^{2} \cdot 5 G_{8}^{2}(q^2) + 2^{7} G_{8}^{3}(q^2) + 2^{8} G_{8}^{4}(q^2)\right)^2\Big).
\end{multline*}
After applying $U_2$ and simplifying, we obtain
\begin{multline*}
\Phi_2(-q)G_{2}^{k+1}(-q)|U_2 \\
= \left(\Phi_2(-q)G_2^{k-1}(-q)|U_2\right)\Big(G_{8}(q) + 2^{2} \cdot 5 G_{8}^{2}(q) + 2^{7} G_{8}^{3}(q) + 2^{8} G_{8}^{4}(q)\Big) + \\
+ \left(\Phi_2(-q)G_{2}^k(-q)|U_2\right)\Big(2^{4} \cdot 3 G_{8}(q) + 2^{6} \cdot 79 G_{8}^{2}(q) + 2^{11} \cdot 83 G_{8}^{3}(q) \\
+ 2^{12} \cdot 659 G_{8}^{4}(q)+ 2^{21} \cdot 11 G_{8}^{5}(q)
+ 2^{23} \cdot 13 G_{8}^{6}(q) + 2^{28} G_{8}^{7}(q) + 2^{28} G_{8}^{8}(q)\Big).
\end{multline*}
Then using the induction hypothesis, we get
\begin{multline}\label{FG2U2modeqinduction}
\Phi_2(-q)G_{2}^{k+1}(-q)|U_2 \\
= \Phi_2(-q^2)\left(\sum_{r=\lceil (k-1)/2\rceil}^{8k-7} c_{1r}^{k-1}G_8^r\right)\Bigg(G_{8}(q) + 2^{2} \cdot 5 G_{8}^{2}(q) + 2^{7} G_{8}^{3}(q) + 2^{8} G_{8}^{4}(q)\Bigg) + \\
+ \Phi_2(-q^2)\left(\sum_{r=\lceil k/2\rceil}^{8k+1} c_{1r}^kG_8^r\right)\Bigg(2^{4} \cdot 3 G_{8}(q) + 2^{6} \cdot 79 G_{8}^{2}(q) + 2^{11} \cdot 83 G_{8}^{3}(q) \\
\qquad\; + 2^{12} \cdot 659 G_{8}^{4}(q)+ 2^{21} \cdot 11 G_{8}^{5}(q)
+ 2^{23} \cdot 13 G_{8}^{6}(q) + 2^{28} G_{8}^{7}(q) + 2^{28} G_{8}^{8}(q)\Bigg)\\
=\Phi_2(-q^2)\sum_{r=\lceil(k+1)/2\rceil}^{8k+9} c_{1,r}^{k+1}G_8^r.
\end{multline}
Using the induction hypothesis for the valuation of coefficients $c_{1r}^{k-1}$ and $c_{1r}^k$, and by carefully estimating valuations of coefficients in front of powers of $G_8$, we obtain
\[
\pi_2(c_{1,r}^{k+1})\ge 2r-k-1.
\]
This completes the induction.
\end{proof}
\begin{lemma}\label{G8Upowers}
For $j\ge1$ there exist integers $t_{jr}$ and $u_{jr}$ such that
\begin{align}
G_8^j|U_2&=\sum_{r=\lceil j/2\rceil}^{2j} t_{jr}G_8^r,\label{G8U2coeffs}\\
\Phi_2(-q)G_8^j|U_2&=\Phi_2(-q^2)\sum_{r=1}^\infty u_{jr}G_8^r,\label{FG8U2coeffs}
\end{align}
where
\begin{align}
\pi_2(t_{jr})&\ge2r-j,\label{G8val}\\
\pi_2(u_{jr})&\ge2r-j\label{FG8val}.
\end{align}
\end{lemma}
\begin{proof}
Formulas~(\ref{G8U2coeffs}) and (\ref{FG8U2coeffs}) follow easily from (\ref{modeq2}) in combination with (\ref{eq:G8U}), (\ref{eq:Phi2U}), and (\ref{eq:Phi2G8U}).

The proof of formulas~(\ref{G8val}) and (\ref{FG8val}) proceeds by induction on $j$. We first prove (\ref{G8val}). From (\ref{eq:G8U}), the statement is true for $j=1$. For $j=2$, apply $U_2$ to (\ref{modeq2}) and use (\ref{eq:G8U}) to get
\begin{multline*}
G_8^2(q)|U_2=(8G_8(q)+32G_8^2(q))(4G_8(q)+16G_8^2(q))+(G_8(q)+4G_8^2(q))\\
=G_8(q)+36G_8^2(q)+256G_8^3(q)+512G_8^4(q).
\end{multline*}
This shows that (\ref{G8val}) is true for $j=2$.
Fix $k\ge2$ and assume the result is true for $j\le k$. We prove that is it true for $j=k+1$. Multiply (\ref{modeq2}) by $G_8^{k-1}$ to obtain
\[
G_8^{k+1}=(8G_8(q^2)+32G_8^2(q^2))G_8^k+(G_8(q^2)+4G_8^2(q^2))G_8^{k-1}.
\]
Apply the $U_2$ operator
\[
G_8^{k+1}|U_2=(8G_8(q)+32G_8^2(q))(G_8^k|U_2)+(G_8(q)+4G_8^2(q))(G_8^{k-1}|U_2).
\]
By (\ref{G8U2coeffs}),
\begin{multline*}
G_8^{k+1}|U_2=(8G_8(q)+32G_8^2(q))\sum_{r=\lceil k/2\rceil}^{2k} t_{kr}G_8^r+(G_8(q)+4G_8^2(q))\sum_{r=\lceil(k-1)/2\rceil}^{2k-2} t_{k-1,r}G_8^r\\
=(8G_8(q)+32G_8^2(q))\sum_{r=\lceil (k-1)/2\rceil}^{2k} t_{kr}G_8^r+(G_8(q)+4G_8^2(q))\sum_{r=\lceil(k-1)/2\rceil}^{2k} t_{k-1,r}G_8^r
\end{multline*}
since $t_{kr}=0$ for $r<\lceil k/2\rceil$ and $r>2k$. We can rewrite that as
\begin{multline}
G_8^{k+1}|U_2=(t_{k-1,\lceil (k-1)/2\rceil}+8t_{k,\lceil (k-1)/2\rceil})G_8^{\lceil(k+1)/2\rceil}\\
+\sum_{r=\lceil(k+3)/2\rceil}^{2k+1}(t_{k-1,r-1}+8t_{k,r-1}+4t_{k-1,r-2}+32t_{k,r-2})G_8^r\\
+(32t_{k,2k}+4t_{k-1,2k})G_8^{2k+2}=\sum_{r=\lceil(k+1)/2\rceil}^{2k+2} t_{k+1,r}G_8^r.
\end{multline}
Thus,
\begin{align*}
t_{k+1,\lceil(k+1)/2\rceil}&=t_{k-1,\lceil (k-1)/2\rceil}+8t_{k,\lceil (k-1)/2\rceil},\\
t_{k+1,r}&=t_{k-1,r-1}+8t_{k,r-1}+4t_{k-1,r-2}+32t_{k,r-2},\\
t_{k+1,2k+2}&=32t_{k,2k}+4t_{k-1,2k},
\end{align*}
and
\begin{align*}
\pi_2(t_{k+1,\lceil(k+1)/2\rceil})&\ge\min\{{\pi_2(t_{k-1,\lceil(k-1)/2\rceil}),\pi_2(8t_{k,\lceil(k-1)/2\rceil})}\}\ge2\lceil(k+1)/2\rceil-k-1,\\
\pi_2(t_{k+1,r})&\ge\min\{{\pi_2(t_{k-1,r-1}),\pi_2(8t_{k,r-1}),\pi_2(4t_{k-1,r-2}),\pi_2(32t_{k,r-2})\}}\\
&\ge\min\{{2r-1-k,2r+1-k,2r-1-k,2r+1-k\}}=2r-k-1,\\
\pi_2(t_{k+1,2k+2})&\ge\min\{{\pi_2(32t_{k,2k}),\pi_2(4t_{k-1,2k})}\}\ge3k+3.
\end{align*}
This proves (\ref{G8val}).

The proof of (\ref{FG8val}) is similar using (\ref{eq:Phi2U}), (\ref{eq:Phi2G8U}), and (\ref{modeq2}).
\end{proof}
Define 
\begin{align}
L_1^s&=\Phi_2(-q)G_2^s(-q)|U_2\\
L_{\alpha}^s&=L_{\alpha-1}^s|U_2
\end{align}
Using Lemma~\ref{G8Upowers} it is easy to see
\begin{align}
L_1^s&=\Phi_2(-q)G_2^s(-q)|U_2=\Phi_2(-q^2)\sum_{j=1}^\infty c_{1j}^sG_8^j\\
L_{2\alpha}^s&=L_{2\alpha-1}^s|U_2=\Phi_2(-q)\sum_{j=1}^\infty b_{\alpha j}^sG_8^j\label{Leven}\\
L_{2\alpha+1}^s&=L_{2\alpha}^s|U_2=\Phi_2(-q^2)\sum_{j=1}^\infty c_{\alpha+1,j}^sG_8^j\label{Lodd}
\end{align}
\begin{theorem}\label{Lval}
\begin{align}
\pi_2(c_{\alpha j}^s)&\ge2\alpha+2j-s-2\\
\pi_2(b_{\alpha j}^s)&\ge2\alpha+2j-s-1
\end{align}
\end{theorem}
\begin{proof}
We will fix $s$ and proceed by induction on $\alpha$. By Theorem~\ref{FG2Upowers}, the statement is true for $\alpha=1$. We will assume the truth of the statement for $c_{\alpha j}^s$ and prove it for $b_{\alpha j}^s$ first. Then we will show the truth of the statement for $c_{\alpha+1,j}^s$. By (\ref{Lodd}),
\[
L_{2\alpha-1}^s=\Phi_2(-q^2)\sum_{j=1}^\infty c_{\alpha j}^s G_8^j.
\]
By (\ref{Leven}) and (\ref{G8U2coeffs}), 
\begin{multline*}
L_{2\alpha}^s=L_{2\alpha-1}^s|U_2=\Phi_2(-q)\sum_{j=1}^\infty c_{\alpha j}^s (G_8^j|U_2)=\Phi_2(-q)\sum_{j=1}^\infty c_{\alpha j}^s\sum_{r=1}^\infty t_{jr}G_8^r\\
=\Phi_2(-q)\sum_{r=1}^\infty G_8^r\sum_{j=1}^\infty c_{\alpha j}^st_{jr}.
\end{multline*}
On the other hand,
\[
L_{2\alpha}^s=\Phi_2(-q)\sum_{r=1}^\infty b_{\alpha r}^s G_8^r.
\]
By compairing the coefficients, we obtain
\[
b_{\alpha r}^s=\sum_{j=1}^\infty c_{\alpha j}^st_{jr}.
\]
Then, by our assumption and (\ref{G8val}),
\[
\pi_2(b_{\alpha r}^s)\ge\min_{j}{\{2\alpha+2j-s-2+2r-j\}}=\min_{j}{\{2\alpha+j-s-2+2r\}}=2\alpha+2r-s-1.
\]
Now, let us prove the statement for $c_{\alpha+1,j}^s$ to complete the induction. By (\ref{Lodd}) and (\ref{FG8U2coeffs}),
\begin{multline*}
L_{2\alpha+1}^s=L_{2\alpha}^s|U_2=\sum_{j=1}^\infty b_{\alpha j}^s(\Phi_2(-q)G_8^j)|U_2\\
=\Phi_2(-q^2)\sum_{j=1}^\infty b_{\alpha j}^s\sum_{r=1}^\infty u_{jr}G_8^r=\Phi_2(-q^2)\sum_{r=1}^\infty G_8^r\sum_{j=1}^\infty b_{\alpha j}^s u_{jr}=\Phi_2(-q^2)\sum_{r=1}^\infty c_{\alpha+1,r}^sG_8^r.
\end{multline*}
Then
\[
c_{\alpha+1,r}^s=\sum_{j=1}^\infty b_{\alpha j}^s u_{jr},
\]
and
\[
\pi_2(c_{\alpha+1,r}^s)\ge\min_j{\{2\alpha+2j-s-1+2r-j\}}=2\alpha+2r-s.
\]
Thus, the induction is complete.
\end{proof}
An easy corollary of Theorem~\ref{Lval} is
\begin{corollary}\label{LvalCorollary}
\begin{align}
L_{2\alpha-1}^s&\equiv0\pmod{2^{2\alpha-s}}\\
L_{2\alpha}^s&\equiv0\pmod{2^{2\alpha-s+1}}
\end{align}
\end{corollary}

\begin{proof}[Proof of Theorem~\ref{thm:OVPinf}]
After applying $U_2^\alpha$ to (\ref{OVPNewman}), and replacing $q$ by $-q$, we obtain
\begin{multline}\label{H1}
\sum_{n=0}^\infty\left(\ell^3 \overline{p}(2^\alpha\ell^2 n)+\ell\left(\dfrac{-2^\alpha n}{\ell}\right)\overline{p}(2^\alpha n)+\overline{p}\left(\dfrac{2^\alpha n}{\ell^2}\right)\right)q^n\\
=(\ell^3+1)\sum_{n=0}^\infty\overline{p}(2^\alpha n)q^n+\sum_{s=1}^{(\ell^2-1)/8}a(s)2^{12s}L_\alpha^s.
\end{multline}

By Corollary~\ref{LvalCorollary},
\begin{align}
2^{12s}L_{2\alpha-1}^s&\equiv0\pmod{2^{2\alpha+11s}},\\
2^{12s}L_{2\alpha}^s&\equiv0\pmod{2^{2\alpha+11s+1}}.
\end{align}
Thus,
\begin{multline*}
\sum_{n=0}^\infty\left(\ell^3 \overline{p}(2^{2\alpha-1}\ell^2 n)+\ell\left(\dfrac{-2^{2\alpha-1} n}{\ell}\right)\overline{p}(2^{2\alpha-1} n)+\overline{p}\left(\dfrac{2^{2\alpha-1} n}{\ell^2}\right)q^n\right)\\
\equiv(\ell^3+1)\sum_{n=0}^\infty \overline{p}(2^{2\alpha-1}n)q^n\pmod{2^{2\alpha+11}},
\end{multline*}
and
\begin{multline*}
\sum_{n=0}^\infty\left(\ell^3 \overline{p}(2^{2\alpha}\ell^2 n)+\ell\left(\dfrac{-2^{2\alpha} n}{\ell}\right)\overline{p}(2^{2\alpha} n)+\overline{p}\left(\dfrac{2^{2\alpha} n}{\ell^2}\right)q^n\right)\\
\equiv(\ell^3+1)\sum_{n=0}^\infty \overline{p}(2^{2\alpha}n)q^n\pmod{2^{2\alpha+12}},
\end{multline*}
and the theorem follows.
\end{proof}

%----- CONCLUSION -----%
\section{Conclusion}
\label{sec:conclusion}
The computations suggest that the conclusion of Theorem~\ref{thm:OVPinf} is not optimal. Empirically, the congruence~(\ref{OVPinf}) appears to persist to higher powers of 2 for every $\alpha$. For example, we observe
\begin{align}\label{OVPexperimental}
\ell^3 \bar{p}(2^2\ell^2 n)+\ell\left(\dfrac{-2^2 n}{\ell}\right)\bar{p}(2^2 n)+\bar{p}\left(\dfrac{2^2 n}{\ell^2}\right)&\equiv (\ell^3+1)\bar{p}(2^2 n)\pmod{2^{20}},\\
\ell^3 \bar{p}(2^3\ell^2 n)+\ell\left(\dfrac{-2^3 n}{\ell}\right)\bar{p}(2^3 n)+\bar{p}\left(\dfrac{2^3 n}{\ell^2}\right)&\equiv (\ell^3+1)\bar{p}(2^3 n)\pmod{2^{21}},\\
\ell^3 \bar{p}(2^4\ell^2 n)+\ell\left(\dfrac{-2^4 n}{\ell}\right)\bar{p}(2^4 n)+\bar{p}\left(\dfrac{2^4 n}{\ell^2}\right)&\equiv (\ell^3+1)\bar{p}(2^4 n)\pmod{2^{24}}.
\end{align}
As James Sellers suggested, this phenomenon may occur only for the initial values of $\alpha$, with the behavior stabilizing for larger $\alpha$. We do not know yet. A systematic analysis remains to be done.

\bibliographystyle{alphaurl}  
\bibliography{references}   

\begin{thebibliography}{FJM05}

\bibitem[CL04]{CorteelLovejoy}
S.~Corteel and J.~Lovejoy.
\newblock Overpartitions.
\newblock {\em Transactions of the American Mathematical Society}, 356(4):1623--1635, 2004.

\bibitem[FJM05]{FJM}
J.-F. Fortin, P.~Jacob, and P.~Mathieu.
\newblock Jagged partitions.
\newblock {\em Ramanujan Journal}, 10(2):215--235, 2005.

\bibitem[GH81]{GordonHughes}
B.~Gordon and K.~Hughes.
\newblock Ramanujan congruences for q(n).
\newblock In Marvin~I. Knopp, editor, {\em Analytic Number Theory}, pages 333--359, Berlin, Heidelberg, 1981. Springer Berlin Heidelberg.

\bibitem[GM25]{GarvanMorrow}
Frank Garvan and Connor Morrow.
\newblock Multiplicative congruences for andrews's even parts below odd parts function and related infinite products.
\newblock {\em Arabian Journal of Mathematics}, 2025.
\newblock \href {https://doi.org/10.1007/s40065-025-00566-4} {\path{doi:10.1007/s40065-025-00566-4}}.

\bibitem[HS05]{HS}
M.~D. Hirschhorn and J.~A. Sellers.
\newblock Arithmetic relations for overpartitions.
\newblock {\em Journal of Combinatorial Mathematics and Combinatorial Computing}, 53:65--73, 2005.

\bibitem[Kno08]{Knopp2008}
Marvin~Isadore Knopp.
\newblock {\em Modular Functions in Analytic Number Theory}.
\newblock AMS Chelsea Publishing, American Mathematical Society, 2008.

\bibitem[Lig75]{Ligozat}
Gerard Ligozat.
\newblock Courbes modulaires de genre 1.
\newblock {\em Mémoires de la Société Mathématique de France}, 43:5--80, 1975.
\newblock URL: \url{http://eudml.org/doc/94716}.

\bibitem[New59]{NewmanEtaQuotient}
Morris Newman.
\newblock Construction and application of a class of modular functions (ii).
\newblock {\em Proceedings of the London Mathematical Society}, s3-9(3):373--387, 1959.
\newblock \href {https://doi.org/10.1112/plms/s3-9.3.373} {\path{doi:10.1112/plms/s3-9.3.373}}.

\end{thebibliography}
%----- BIBLIOGRAPHY -----%
%\printbibliography

\end{document}